\documentclass[a4paper,10pt]{article}
\usepackage[utf8x]{inputenc}
\usepackage{fullpage}
\usepackage{amsmath, amsthm}
\usepackage{amsfonts}

\title{Upper bound for the Laplacian eigenvalues of a graph}
\author{Miriam Farber and Ido Kaminer, Technion-Israel Institute of Technology}
\date{}
\newtheorem{theorem}{Theorem}
\newtheorem{corollary}{Corollary}

\begin{document}
\maketitle
\subsection*{Abstract}
In this note we give a new upper bound for the Laplacian eigenvalues of an unweighted graph. Let $G$ be a simple graph on $n$ vertices. Let $d_{m}(G)$ and $\lambda_{m+1}(G)$ be the $m$-th smallest degree of $G$ and the $m+1$-th smallest Laplacian eigenvalue of $G$ respectively. Then $ \lambda_{m+1}(G)\leq d_{m}(G)+m-1 $ for $\bar{G} \neq K_{m}+(n-m)K_1 $. We also introduce upper and lower bound for the Laplacian eigenvalues of weighted graphs, and compare it with the special case of unweighted graphs.

\begin{center}
\subsubsection*{\emph{This note is dedicated to Professor Abraham Berman in recognition of his important contributions to linear algebra and for his inspiring guidance and encouragement}}
\end{center}

\subsection*{1. Introduction}
Let $G=\big(E(G),V(G)\big)$ be a simple graph with $|V(G)|=n$. We say that $G$ is a \emph{weighted graph} if it has a weight (a positive number) associated with each edge. The weight of an edge $ \big\{i,j\big\} \in E(G) $ will be denoted by $ w_{ij} $. We define the \emph{adjacency matrix} $ A(G) $ of $G$ to be a symmetric matrix which satisfies
\begin{equation*}
  a_{ij}=\left\{
    \begin{array}{cl}
      0 & \text {if } \big\{i,j\big\} \notin E(G) \\
      w_{ij} & \text {if } \big\{i,j\big\} \in E(G)
  \end{array} \right.
\end{equation*}
The degree diagonal matrix of $G$ will be denoted by $D(G)=diag\big(d_1(G),d_2(G),\ldots ,d_n(G)\big)$ where\\ 
$d_i(G) =\displaystyle{  \sum_{j} a_{ij}} $ is the degree of the vertex $ v_i $ , and we assume without loss of generality that\\ 
$d_1(G) \leq d_2(G) \leq \ldots \leq d_n(G)$.
We denote by $ L(G)=D(G)-A(G) $ the Laplacian matrix of $ G $, and its eigenvalues
arranged in increasing order: $ 0=\lambda_1(G) \leq \lambda_2(G) \leq \ldots \leq \lambda_n(G) $. The largest and the smallest degrees of $G$ will be denoted by $ \bigtriangleup(G) $ and $ \delta(G) $ respectively.\\
It is well known, e.g. \cite{Mohar}, that for a simple unweighted graph $G$, the eigenvalues of $G$, and of its complement graph $\bar{G}$, are related in the following way:
\begin{align*}
(*) \hspace{1.5 cm} \lambda_i(\bar{G}) &=n-\lambda_{n-i+2}(G) \\
\end{align*}
There are several known lower bounds on the Laplacian eigenvalues of unweighted graph. Grone and Merris \cite{Grone_Merris} proved that $\lambda_n(G) \geq d_n(G)+1$. Li and Pan \cite{Li_Pan} showed that $\lambda_{n-1}(G) \geq d_{n-1}(G)$. And finally, Guo\cite{Guo} gave a lower bound on the third largest Laplacian eigenvalue: $\lambda_{n-2}(G) \geq d_{n-2}(G)-1$. Brouwer and Haemers generalized these bounds in the following theorem\cite{Brouwer_Haemers}:
\begin{theorem}
    Let $G$ be finite simple unweighted graph on $n$ vertices. If $G$ is not $K_{n-i+1}+(i-1)K_1$, then $\lambda_i(G) \geq d_i(G)-n+i+1$.
\end{theorem}

\subsection*{2. The unweighted case}
The main result in this note is:
\begin{theorem}
    Let $G$ be finite simple unweighted graph on $n$ vertices. If $\bar{G} \neq K_{m}+(n-m)K_1 $, then $ \lambda_{m+1}(G)\leq d_{m}(G)+m-1 $.
\end{theorem}

\begin{proof}
    The degree sequence of $\bar{G}$ is $n-1-d_n(G) \leq n-1-d_{n-1}(G) \leq \ldots \leq n-1-d_1(G)$. From this it follows by Theorems 1 and (*) that:
    \begin{align}
    \lambda_{n-m+1}(\bar{G})                 & =  n-\lambda_{m+1}(G)  \\
    \lambda_{n-m+1}(\bar{G})                  & \geq d_{n-m+1}(\bar{G})-m+2  \text{  for  }  \bar{G} \neq K_m+(n-m)K_1 \\
    d_{n-m+1}(\bar{G})                           & = n-1-d_m(G)
    \end{align}
    From (1) and (2) we get $  n-\lambda_{m+1}(G)\geq d_{n-m+1}(\bar{G})-m+2 $. Combining it with (3) yields:\\ $ n-\lambda_{m+1}(G)\geq n-1-d_m(G)-m+2 $. Therefore $ \lambda_{m+1}(G)\leq d_{m}(G)+m-1 $ for $\bar{G} \neq K_{m}+(n-m)K_1 $.
\end{proof}

One of the consequences of this theorem, is that now we can describe the intervals in which the Laplacian eigenvalues lies.

\begin{corollary}
    Let $G$ be finite simple unweighted graph on $n$ vertices, and suppose $G \neq K_{n-m+1}+(m-1)K_1 $ and $ \bar{G} \neq K_{m-1}+(n-m+1)K_1 $. Then $ d_m(G)-n+m+1\leq \lambda_{m}(G)\leq d_{m-1}(G)+m-2 $.\\
    In addition, if $ d_m(G)= d_{m-1}(G)+n-3 $, then the inequality in Theorems 1 and 2 becomes equality.
\end{corollary}

Example : Let $ G $ be a graph that is constructed from $ K_{1,n-1} $ by adding an edge (or edges) between pairs of unconnected vertices of $ K_{1,n-1} $, such that $ d_{n-1}(G)=2 $. Then equality holds for $ \lambda_{n}(G) $ .

\begin{proof}
    The first part follows immediately from Theorems 1 and 2. Now, if $ d_m(G)= d_{m-1}(G)+n-3 $, then  $ d_m(G)-n+m+1= d_{m-1}(G)+n-3-n+m+1 =d_{m-1}(G)+m-2$, hence the inequalities becomes equalities.
\end{proof}

\subsection*{3. The weighted case}
Let $G$ be a simple weighted graph. Let $l_m=\big\{v_m,v_{m+1}, \ldots , v_n \big\}$ be the set of $n-m+1$ vertices with the largest degrees in $G$, and let $S_m=\big\{v_1,v_{2}, \ldots , v_m \big\}$ be the set of $m$ vertices with the smallest degrees in $ G $. Define $G_{l_m}$ and $G_{S_m}$ be the subgraphs induced from $G$ by $l_m$ and $S_m$ respectively. Now we are ready to prove the following:
\begin{theorem}
Let $ G $ be a simple weighted graph. Then $ \lambda_{m}(G) \geq d_m(G)- \bigtriangleup(G_{l_m})$
\end{theorem}

\begin{proof}
We will use the following well known result about eigenvalues, e.g \cite{Horn_Johnson}.

\begin{center}(**)
     Let $A$ be a symmetric matrix. Then $ \lambda_{k}(A) = max\big\{\frac{\langle Af,f \rangle}{\langle f,f \rangle} | f \bot f_{k+1},f_{k+2}, \ldots f_n\big\} $ ,when $ f_{k+1},f_{k+2}, \ldots f_n $ are eigenvectors of the eigenvalues $ \lambda_{k+1},\lambda_{k+2}, \ldots \lambda_n $ respectively.
\end{center}

    Let $ f_{m+1},f_{m+2}, \ldots f_{n} $ be the eigenvectors of the eigenvalues $ \lambda_{m+1},\lambda_{m+2}, \ldots \lambda_{n} $ respectively. We define a vector $g$ in the following way:

   \begin{equation*}
    g_j=\left\{
    \begin{array}{cl}
      0 & \text {if } j \notin l_m \\
      c_{v_j} & \text {if } v_j \in l_m
    \end{array} \right.
    \end{equation*}
    ,where $ \big\{ c_{v_m},c_{v_{m+1}}, \ldots ,c_{v_n} \big\} $ is a set of variables such that $ g \bot f_i $ for all $ m+1 \leq i \leq n $ . Note that since we have $ n-m $ equations, with $ n-m+1 $ variables, there exists such nontrivial set (i.e. not all the variables are zero).From (**) it follows:
    \begin{align}
    \lambda_{m}(G) & = max\big\{\frac{\langle Lf,f \rangle}{\langle f,f \rangle} | f \bot f_{m+1},f_{m+2}, \ldots f_n\big\} \geq \frac{\langle L(G)g,g \rangle}{\langle g,g \rangle} \nonumber \\
    & = \frac{\langle D(G)g,g \rangle}{\langle g,g \rangle}-\frac{\langle A(G)g,g \rangle}{\langle g,g \rangle} \nonumber \\
    & = \frac{\displaystyle{  \sum_{j=m}^{n} d_j(G)c_{v_j}^2}}{\displaystyle{  \sum_{j=m}^{n} c_{v_j}^2}}-\frac{\langle A(G)g,g \rangle}{\langle g,g \rangle} \nonumber \\
    \end{align}

    Now, from (**) and the construction of $g$, $ \frac{\langle A(G)g,g \rangle}{\langle g,g \rangle} $ is smaller than or equal to the largest eigenvalue of $ A(G_{l_m}) $, which, according to the Gersgorin Disc Theorem (e.g \cite{Horn_Johnson}), is smaller than or equal to $\bigtriangleup(G_{l_m})$.
    Finally, from (4) we get $ \lambda_{m}(G) \geq d_m(G)- \bigtriangleup(G_{l_m})$
\end{proof}

Our goal now is to find an upper bound on the eigenvalues of weighted graph.
Let $ a $ be the maximal weight of an edge in $ G $. We normalize $G$ by dividing its weights by $a$, and denote the normalized graph by $\hat{G}$. The complement graph of $ \hat{G} $, $ \bar{\hat{G}} $ is defined in the following way: $ V(\hat{G})=V(\bar{\hat{G}}) , w_{ij}(\bar{\hat{G}})=1-w_{ij}(\hat{G}) $. For $ \hat{G} $ and $ \bar{\hat{G}} $ , (*) holds (e.g. \cite{Mohar}).

We are ready now to present the upper bound for the Laplacian eigenvalues of weighted graphs:

\begin{theorem}
Let $ G $ be a simple weighted graph, and let $ a $ be the maximal weight of an edge in $ G $.
Then $ d_{m}(G) + ma - \delta(G_{S_m})  \geq  \lambda_{m+1}(G)  $.
\end{theorem}
\begin{proof}
    The proof is very similar to the proof of theorem 2. First, we define the graph $ \hat{G} $ , which is constructed from G as we explained above. We notice that $S_m$ is the set of $ m $ vertices with the largest degrees in $ \bar{\hat{G}} $. The graph H would be the graph that induced from $ \bar{\hat{G}} $ by $S_m$. Prepositions (1), (3) (from the proof of Theorem 2) still holds for $ \hat{G} $, and instead of (2), we can write, using Theorem 3,
    \begin{align*}
    \lambda_{n-m+1}(\bar{\hat{G}})\geq d_{n-m+1}(\bar{\hat{G}}) - \bigtriangleup(H).
    \end{align*}
    By following the steps from the proof of Theorem 2, we get
    \begin{align*}
    d_{m}(\hat{G}) + 1 + \bigtriangleup(H) \geq  \lambda_{m+1}(\hat{G}).
    \end{align*}
    and since
    \begin{align*}
      \bigtriangleup(H) = m-1 - \delta(G_{S_m})
     \end{align*}
     we have
     \begin{align*}
     d_{m}(\hat{G}) + m - \delta(G_{S_m})  \geq  \lambda_{m+1}(\hat{G}).
     \end{align*}
     We conclude that
     \begin{align*}
     d_{m}(G) + ma - \delta(G_{S_m})  \geq  \lambda_{m+1}(G)
     \end{align*}
\end{proof}

\begin{corollary}
    Let $G$ be finite simple weighted graph on $n$ vertices, and denote by $ a $ the maximal weight of an edge in $ G $ . Then $   d_m(G)- \bigtriangleup(G_{l_m}) \leq  \lambda_{m}(G) \leq d_{m-1}(G) + (m-1)a - \delta(G_{S_{m-1}}) $.
\end{corollary}
    In the special case of unweighted graphs, the bounds reduce to  $   d_m(G)- n+m  \leq  \lambda_{m}(G) \leq d_{m-1}(G) + m-1 $. It is interesting to compare it with the stronger result from Section 2: $ d_m(G)-n+m+1\leq \lambda_{m}(G)\leq d_{m-1}(G)+m-2 $.

\bibliographystyle{plain}

\end{document}